\documentclass[reqno]{amsart}

\usepackage{amssymb,amsfonts,amstext,amsthm,hyperref,cleveref,xcolor}

\usepackage{graphics}
\usepackage{graphicx}
\usepackage{epstopdf}% To incorporate .eps illustrations using PDFLaTeX, etc.
\usepackage{indentfirst}
\usepackage{float}
\usepackage[numbers,sort&compress]{natbib}% Citation support using natbib.sty
\bibpunct[, ]{[}{]}{,}{n}{,}{,}% Citation support using natbib.sty
\makeatletter% @ becomes a letter
\def\NAT@def@citea{\def\@citea{\NAT@separator}}% Suppress spaces between citations using natbib.sty
\makeatother% @ becomes a symbol again
\theoremstyle{plain}% Theorem-like structures provided by amsthm.sty
\newtheorem{theorem}{Theorem}[section]

\newtheorem{corollary}[theorem]{Corollary}

\crefrangeformat{equation}{#3(#1)#4--#5(#2)#6}
\crefname{enumi}{\unskip}{\unskip}

\theoremstyle{definition}

\newtheorem{remark}[theorem]{Remark}

\begin{document}

\title[How to construct an upper triangular matrix]{How to construct a upper triangular matrix that satisfy the quadratic polynomial equation with different roots}

\author{Ivan Gargate}
\address{UTFPR, Campus Pato Branco, Rua Via do Conhecimento km 01, 85503-390 Pato Branco, PR, Brazil}
\email{ivangargate@utfpr.edu.br}

\author{Michael Gargate}
\address{UTFPR, Campus Pato Branco, Rua Via do Conhecimento km 01, 85503-390 Pato Branco, PR, Brazil}
\email{michaelgargate@utfpr.edu.br}

\begin{abstract}
Let $R$ be an associative ring with identity $1$. We describe all matrices in $T_n(R)$ the ring of $n\times n$ upper triangular matrices over $R$ ($n\in \mathbb{N}$), and $T_{\infty}(R)$ the ring of infinite upper triangular matrices over $R$, satisfying the quadratic polynomial equation $x^2-rx+s=0$. For such propose we assume that the above polynomial have two different roots in $R$. Moreover, in the case that $R$ in finite, we compute the number of all matrices to solves the matrix equation $A^2-rA+sI=0,$ where $I$ is the  identity matrix.
\end{abstract}


\maketitle

\section{Introduction}\label{intro}
Let $R$ be an associative ring with identity 1. Denote by $T_n(R)$ the $n\times n$ upper triangular groups with entries in $R$ and $T_{\infty}(R)$ the ring of infinite upper triangular matrices over $R$. There are several authors who have works over this spaces, for instance, Slowik \cite{Slowik} show how to construct an involution matrix over these spaces, Hou \cite{Hou} proof the similar results for idempotent matrices and Gargate in \cite{Gargate} compute the number off all involutions over the incidence algebras $\mathcal{I}(X,\mathbb{K})$ where $X$ is a finite poset and $\mathbb{K}$ is a finite field. Recentely Gargate \cite{Gargate2} compute the number of coninvolution matrices over the special rings: the Gaussian Integers module $p$ and the Quartenion Integers module $p$, with $p$ an odd prime number. Remember that various special matrices satisfy some polynomial equations, for instance,  idempotent matrices satisfies $x^2-x=0$ and  involution matrices satisfies $x^2-1=0$. 

In the present article the authors generalizes the results of \cite{Slowik} and \cite{Hou} on a broader class of matrices that satisfy the polynomial equations $x^2-rx+s=0$ with the condition that the polynomial has two different roots in $R$.  We investigate how to construct these special matrices and compute the total of these matrices when $R$ is a finite ring.

Our main results is the followings Theorem:

%Let $R$ be an associative ring with identity 1. An $k$-potent of an associative ring $R$  is an element $g\in R$ satisfying $g^k=g$ for $k\geq 2\in \mathbb{N}$, if $k=2$ then the element $g$ is called of idempotent. In the case of idempotent matrix, there are several authors
%who have works in this respect, for example expressing various matrices as a difference, a sum and product of some idempotents, for instance  see \cite{Har, Har2, Slowik3, Tang, Wu, Ala, Bas, Bot, Erdo, Fac, Laf}. Recently Hou in \cite{Hou}  describe the necessary and sufficient conditions on the infinite and finite upper triangular matrices over a unitary ring to be idempotents with only zeros and ones on the diagonal. 

%In the case of $k$-potent matrix the authors in \cite{Gargate} generalize the article of Wu \cite{Wu} and show the conditions by which  a complex matrix can expressed as a sum of $k$-potents matrices.
 
%We investigate how to construct and count $k$-potent matrices on upper triangular matrices over a associative ring with identity generalizing the recent work of Hou \cite{Hou} on idempotent triangular matrices having only ones and zeros on their main diagonal.

% In this article, we described the necesary and sufficient conditions on infinite and finite upper triangular matrices over a unitary ring to be $(k+1)$-potents with set elements $\{0, 1, \omega, \omega^2, \ldots, \omega^{k-1}\}$ on the diagonal, where $\omega$ is a $kth$ root of unity.

%Our main results is the followings Theorem:
 
\begin{theorem}\label{th1} Assume that $R$ is an associative ring with identity 1. Let $M$ be either the
group $T_n(K)$ or $T_{\infty}(K)$  for some $n \in\mathbb{N}$ and denote by $I$  the identity matrix of $M$. Consider the quadratic polynomial equation $x^2-rx+s=0$ and assume that this equation has two different roots $a,b\in R$ such that $a-b$ is not a right zero divisor. Then a matrix $A\in M$ satisfies the quadratic equation of the type 
\begin{equation}\label{eq1}
A^2-r A+ s I=0,
\end{equation}
    if and only if $A$ is
described by the following statements:
\begin{itemize}
\item[(i)] For all $1\leq i\leq n$, we have  $a_{ii}\in\{a,b\}$ with $a, b$ different roots of the quadratic equation $x^2-rx+s=0$, $r=a+b$ and $s=ab$.
\item[(ii)] For all pairs of indices $1\leq i< j\leq n$ such that $a_{ii} = a_{jj}$, then $a_{ij}$ equals to
\begin{equation}\label{eqth}
a_{ij}=\left\{\begin{array}{lr} 
0\ \ & \text{if}\ \ j=i+1\\
 -\displaystyle\frac{1}{a_{ii}-{\text{another \ root}}}\displaystyle\sum_{p=i+1}^{j-1}a_{ip}a_{pj} 
&\text{if}\  \ j > i +1. 
 \end{array}\right.
\end{equation}
\item[(iii)] For $i < j$, such that $a_{ii} \not = a_{jj}$, then $a_{ij}$ can be chosen arbitrarily.
\end{itemize}

\end{theorem}

Next using the above theorem we will prove the following result

\begin{theorem}\label{th2} Let $R$ be an associative ring with identity 1 and $|R|=q$ the number of the elements
in $R$. Consider the quadratic polynomial equation $x^2-rx+s=0$ and assume that this equation has two different roots $a,b\in R$ such that $a-b$ is not a right zero divisor. Then the total number of $ n\times n$ upper triangular matrices that satisfy the quadratic equaton $A^2-rA+sI=0$  is equal to
$$ \displaystyle \sum_{\substack{n_1+n_1=n \\ 0\leq n_i}} \binom{n}{n_1n_2}\cdot q^{n_1n_2}.
$$
where $n_1, n_2$ are the number of times  that appears $a,b$ in the diagonal respectively.
\end{theorem}

\section{Matrix solutions of the equation $A^2-r A+s I=0$}

We star our considerations we notice the following property.

\begin{remark}\label{lemma1} Assume that $R$ is an associative ring with identity $1$, $M=T_{\infty}(R)$ or $M=T_n(R)$ for some $n\in\mathbb{N}$. If $A\in M$ is a block matrix such that
$$A=\left[ \begin{array}{cccc}
B_{11} &B_{12}& B_{13}& \cdots \\ 
&B_{22}&B_{23} & \cdots\\
 && B_{33}& \cdots  \\
 &&&\ddots
\end{array}
\right]
$$
where $B_{ii}$ are square  matrices and $A$ satisfies the quadratic equation $A^2-r A+s I=0$, then for all $i$, the matrices $B_{ii}$ satisfy the quadratic equation as well.
\end{remark}
\begin{proof} Since $A$ satisfies the quadratic equation $A^2-r A+s I=0$, we have
$$
\begin{array}{rcl}
     A^2\!-\! rA\!+\! s I\!\!\! &\! =\!&\!\!\!\left[ \begin{array}{cccc}
B_{11}^{2}- r B_{11}+s I & *& *& \cdots \\ 
&B_{22}^{2}- rB_{22}+s I&* & \cdots\\
 && B_{33}^{2}- r B_{33}+s I& \cdots  \\
 &&&\ddots
\end{array}
\right] 
\end{array},
$$
and we obtain $B_{ii}^{2}-r B_{ii}+s I=0 $ for all $i$ by comparing entries of the diagonal position in the matrix equality above.
\end{proof}

Now, we can prove our first main result.

%\begin{theorem}\label{teoprincipal} Assume that $R$ is an associative ring with identity 1. Let $M$ be either the
%group $T_n(K)$ or $T_{\infty}(K)$  for some $n \in\mathbb{N}$ and $I$  the identity matrix. Then a matrix $A\in M$ satisfies a quadratic equation of the type
%\begin{equation}
%A^2-r A+ s I=0,
%\end{equation}
 %   if and only if $A$ is
%described by the following statements:
%\begin{itemize}
%\item[(i)] For all $1\leq i\leq n$, we have  $a_{ii}\in\{a,b\}$ with $a, b$ different roots of the quadratic equation $x^2-rx+s=0$, $r=a+b$ and $s=ab$.
%\item[(ii)] For all pairs of indices $1\leq i< j\leq n$ such that $a_{ii} = a_{jj}$, then $a_{ij}$ equals to
%\begin{equation}
%a_{ij}=\left\{\begin{array}{lr} 
%0\ \ & \text{if}\ \ j=i+1\\
% -\displaystyle\frac{1}{a_{ii}-{\text{another \ root}}}\displaystyle\sum_{p=i+1}^{j-1}a_{ip}a_{pj} 
%&\text{if}\  \ j > i +1. 
% \end{array}\right.
%\end{equation}
%\item[(iii)] For $i < j$, such that $a_{ii} \not = a_{jj}$, then $a_{ij}$ can be chosen arbitrarily.
%\end{itemize}

%\end{theorem}
\begin{proof}[Proof of Theorem \ref{th1}] Let $A=\sum_{i,j}a_{ij}E_{ij}\in M$ be a matrix that satisfies the equation (\ref{eq1}). As we have that $a$ and $b$ are different roots of the quadratic equation then $r=a+b$ and $s=ab$.

Since $A^2-rA+sI=0$ our coefficients must satisfy the equations:

\begin{equation}\label{eq3}
\left\{\begin{array}{l}
a_{ii}^2-r\cdot a_{ii}+s=0,\\
\\
a_{ii}a_{i,i+i}+a_{i,i+1}a_{i+1,i+1}-r\cdot a_{i,i+1}=0,\\
\\
a_{ii}a_{i,i+2}+a_{i,i+1}a_{i+1,i+2}+a_{i,i+2}a_{i+2,i+2}-r\cdot 
a_{i,i+2}=0,\\
\\
\ \ \ \ \ \vdots\\
\\
\displaystyle\sum_{p=0}^{m}a_{i,i+p}a_{i+p,i+m}- r\cdot a_{i,i+m}=0,\\
\\
\ \ \ \ \ \vdots
\end{array}\right.
\end{equation}

Since $A$ satisfies the equation (\ref{eq1})  and $a_{ii}^2-r\cdot a_{ii}+s=0$ then $a_{ii}\in\{a,b\}$.

We need to proved that (ii) and (iii) given in Theorem \ref{th1} hold. We use induction on $j-i$.

Assume that $j-i=1$. We have
\begin{equation}\label{eq4}
a_{ii}a_{i,i+i}+a_{i,i+1}a_{i+1,i+1}-r\cdot a_{i,i+1}=0,
\end{equation}
from the family of equations (\ref{eq3}). One can see that:
\begin{itemize}
    \item If $a_{ii}=a_{i+1,i+1}$,  of the equation (\ref{eq4}) we have
    $$2a_{ii}a_{i,i+1}-r\cdot a_{i,i+1}=0,$$
or
$$a_{i,i+1}\left(2a_{ii}-r\right)=0,
$$
then $a_{i,i+1}=0$ since $a_{ii}\in\{a,b\}$ and $r=a+b$ with $a\not= b$.
   \item If $a_{ii}\not=a_{i+1,i+1},$ then of the equation (\ref{eq4}) we obtain
   $$a_{i,i+1}\left(a_{ii}+a_{i+1,i+1}-r\right)=0,
   $$
thus $a_{i,i+1}$ can be chosen arbitrarily, since $r=a+b=a_{ii}+a_{i+1,i+1}$. 
\end{itemize}
So the first super diagonal entries of the matrix $A$ fulfill (ii) and (iii).

Now, suposse that $j-i-=m>1$ and consider the $(i,i+m)$ entries of the equation (\ref{eq1}), and we have the $(m+1)$-st family of the equation (\ref{eq3}):

$$\displaystyle\sum_{p=0}^{m}a_{i,i+p}a_{i+p,i+m}- r\cdot a_{i,i+m}=0,
$$
or \begin{equation}\label{eq5}
a_{i,i+m}\left(a_{ii}+a_{i+m,i+m}-r\right)+\sum_{p=1}^{m-1}a_{i,i+p}a_{i+p,i+m}=0.
\end{equation}

\begin{itemize}
\item If $a_{ii}=a_{i+m,i+m}$ then $\left(a_{ii}+a_{i+m,i+m}-r\right)\not= 0$ and  we obtain that
\begin{equation}\label{eqth2}
\begin{array}{rcl}
a_{i,i+m}&=&-\displaystyle\frac{1}{(a_{ii}+a_{i+m,i+m}-r)}\displaystyle\sum_{p=1}^{m-1}a_{i,i+p}a_{i+p,i+m}\\
&&\\\
&=&-\displaystyle\frac{1}{(a_{ii}-\textit{other root})}\displaystyle\sum_{p=1}^{m-1}a_{i,i+p}a_{i+p,i+m}
\end{array}
\end{equation}
where $r=a+b$ and
$$
\left(a_{ii}+a_{i+m,i+m}-r\right)=\left(a_{ii}-\textit{other\ root}\right)=\left\{\begin{array}{cc}
  a-b   &, \ if \ \ a_{ii}=a_{i+m,i+m}=a  \\
  &\\
  b-a  & ,\  if \ \ a_{ii}=a_{i+m,i+m}=b
\end{array}\right.
$$

So (ii) of the Theorem \ref{th1} hold.

\item If $a_{ii}\not= a_{i+m,i+m}$ then we must have $a_{ii}+a_{i+m,i_m}-r=0$ since $a_{ii}\in\{a,b\}$ and $r=a+b$. So we get \begin{equation}\label{eqq}\displaystyle\sum_{p=1}^{m-1}a_{i,i+p}a_{i+p,i+m}=0
\end{equation}
from equation (\ref{eq5}).

Now, consider $A(m,i)$ the submatrix of $A$ defined as
$$
A(m,i)=\left[
\begin{array}{cccc}
a_{ii} & a_{i,i+1}& \cdots & a_{i,i+m}\\
  &  a_{i+1,i+1} & \cdots & a_{i+1,i+m}\\
  & & \ddots &\vdots\\
  &&&a_{i+m,i+m}
\end{array}
\right].
$$

From Remark (\ref{lemma1}) one can see that $A$ satisfies the quadratic equation (\ref{eq1})  if and only if $A(m,i)$ also satisfies the equation (\ref{eq1}) for all $m$ and $i$.

We write this matrix as a block matrix such that

\begin{equation}\label{eq6}
A(m,i)=\left[
\begin{array}{ccc}
a_{ii} & \alpha&  a_{i,i+m}\\
 0 &  \beta & \gamma\\
 0 &0&a_{i+m,i+m}
\end{array}
\right].
\end{equation}

Since $A$ satisfies the equation (\ref{eq1}) and by Remark (\ref{lemma1}) we have  that the matrices 
$$A(m-1,i)=\left[
\begin{array}{cc}
a_{ii} & \alpha\\
 0 &  \beta 
\end{array}
\right] \ \ \ and  \ \ \ A(m-1,i+1)=\left[
\begin{array}{cc}
    \beta & \gamma\\
 0 &a_{i+m,i+m}
\end{array}
\right],
$$ 
also satisfies the equation (\ref{eq1}). So, we obtain that
$$ a_{ii}\alpha+\alpha\beta - r \alpha=0
$$
and 
$$\beta\gamma+\gamma a_{_{i+m,i+m}}- r \gamma=0.
$$

Thus
$$
\begin{array}{rl}
   \left(A(m,i)\right)^2\!-\! r A(m,i)\!+\!s I \!\!\!&\!\!=\!\!\!\left[\!
\begin{array}{ccc}
0 \!&\!  0 &a_{ii}a_{i,i+m}+\alpha\gamma+a_{i,i+m}a_{i+m,i+m}-ra_{i,i+m} \\
0 \! &\!  0 & 0\\
 0&0&0
\end{array}
\!\right]
\end{array}
$$
since $a_{ii}\in\{a,b\}$ with $a, b$ roots of the equation $x^2-rx+s=0$.

As $a_{ii}\not= a_{i+m,i+m}$ from equations (\ref{eq5}) and (\ref{eqq}) we have
$$\alpha\gamma=\displaystyle\sum_{p=1}^{m-1}a_{i,i+p}a_{i+p,i+m}=0.
$$

Hence,
$$
\begin{array}{rcl}
a_{ii}a_{i,i+m}+\alpha\gamma+a_{i,i+m}a_{i+m,i+m}-ra_{i,i+m}\!\!&\!=\!&\!\!a_{i,i+m}(a_{ii}+a_{i+m,i+m}-r)+\alpha\gamma\\   
\!\!&\!=\!&\!\! \alpha\gamma\\
\!\!&\!=\!&\!\!0
\end{array}
$$
since $r=a+b=a_{ii}+a_{i+m,i+m}$.

Therefore, $A(m,i)$ satisfies the equation (\ref{eq1}), regardless of the value of the entry $a_{i,i+m}$.Thus (iii) of the Theorem \ref{th1} holds.
\end{itemize}

Assume now that the entries of $A$ fulfill (i), (ii) and (iii) of Theorem \ref{th1}. We shall prove that $A$ satisfies the quadratic equation $A^2-rA+sI=0$. Since the equation (\ref{eqth}) involves only the coefficients with indices  $p$, such that $i\leq p \leq j$, it suffices to prove the claim for $A(m,i)$. For $m=1$ and $m=2$ one can easly check now that all sub matrices $A(1,i)$ and $A(2,i)$ satisfy the quadratic equation $A^2-rA+sI=0$. Suppose that the claim hold for all $1\leq t \leq m-1$, i.e. $A(2,i), A(3,i),\ldots, A(m-1,i)$ satisfy the quadratic equation (\ref{eq1}) for all $i$, we need only prove that $A(m,i)$ also satisfy the quadratic equation (\ref{eq1}).

Consider $A(m,i)$ as a block matrix given in the form of equation (\ref{eq6}). Thus,  we have that the quadratic equation $A(m,i)^2-rA(m,i)+sI$ equals

\begin{equation}\label{eq7}
\left[\!
\begin{array}{ccc}
a_{ii}^2-ra_{ii}+s &  a_{ii}\alpha+\alpha\beta-r\alpha &a_{ii}a_{i,i+m}\!+\!\alpha\gamma+a_{i,i+m}a_{i+m,i+m}-ra_{i,i+m} \\
0  &  \beta^2-r\beta+s & \beta\gamma+\gamma a_{i+m,i+m}-r\gamma\\
 &&a_{i+m,i+m}^2-ra_{i+m,i+m}+s
\end{array}\!
\right]
\end{equation}

By assumption, $A(m-1,i)$ satisfy the equation (\ref{eq1}) for all $i$. So $A(m-1,i)$ and $A(m-1,i+1)$ satisfy the equation (\ref{eq1}). Thus,

\begin{equation}\label{eq9}
\begin{array}{rcl}
A(m-1,i)^2-rA(m-1,i)+sI&=&\left[
\begin{array}{cc}
a_{ii}^2-ra_{ii}+s  &  a_{ii}\alpha+\alpha\beta-r\alpha\\
 0 & \beta^2-r\beta+s 
\end{array}
\right]\\
&&\\
&=&\left[
\begin{array}{cc}
    0 & 0\\
 0 &0
\end{array}
\right],
\end{array}
\end{equation} 
and
\begin{equation}\label{eq10}
\begin{array}{rcl}
A(m\!-\!1,i\!+\!1)^2\!-\!rA(m\!-\!1,i\!+\!1\!)\!+\!sI\!\!\!\!&\!=\!&\!\!\!\!\left[\!\!
\begin{array}{cc}
\beta^2-r\beta+s \!\! &\!\!  \beta\gamma+a_{i+m,i+m}\gamma-r\gamma\\
 0\!\! &\!\! a_{i+m,i+m}^2-ra_{i+m,i+m}+s 
\end{array}
\!\!\right]\\
&&\\
\!\!\!\!&\!=\!&\!\!\!\!\left[
\begin{array}{cc}
    0 & 0\\
 0 &0
\end{array}
\right].
\end{array}
\end{equation} 

From the equations (\ref{eq9}) and (\ref{eq10}) above, we have $a_{ii}^2-ra_{ii}+s=0$ and $a_{i+m,i+m}^2-ra_{i+m,i+m}+s=0 $ since $a_{ii}\in \{a,b\}$ is root the equation $x^2-rx+s=0$,  $\beta^2-r\beta+sI=0$ by Lemma (\ref{lemma1}) and

\begin{equation}\label{eq11}
    a_{ii}\alpha+\alpha\beta-r\alpha=0 
\end{equation}

\begin{equation}\label{eq12}
    \beta\gamma+a_{i+m,i+m}\gamma-r\gamma=0
\end{equation}

Then by multiplying the equation (\ref{eq11})  by $\gamma$ and the equation (\ref{eq12}) by $\alpha$ we obtain that
$$
a_{ii}\alpha\gamma+\alpha\beta\gamma-r\alpha\gamma=0 
$$
$$
\alpha\beta\gamma+a_{i+m,i+m}\alpha\gamma-r\alpha\gamma=0.
$$
Hence,
$$\alpha\beta\gamma=(r-a_{ii})\alpha\gamma
$$
$$\alpha\beta\gamma=(r-a_{i+m,i+m})\alpha\gamma.
$$

\begin{itemize}
\item If we consider $a_{ii}\not=a_{i+m,i+m}$ we have
$$\alpha\beta\gamma=(a_{ii})\alpha\gamma=(a_{i+m,i+m})\alpha\gamma,
$$
since $r=a_{ii}+a_{i+m,i+m}$, which implies that $\alpha\gamma=0$.

So, the $(1,m)$ entries of equation (\ref{eq7}) is
$$a_{i,i+m}\left(a_{ii}+a_{i+m,i+m}-r\right)+\alpha\gamma = \alpha\gamma=0.$$

Therefore, $A(m,i)$ satisfies the quadratic equation $A^2-rA+sI=0$.

\item On the other hand, if $a_{ii}=a_{i+m,i+m}$ from (iii) of the  Theorem \ref{th1} or the equations (\ref{eq5}) and  (\ref{eqth2}) we have
$$
a_{i,i+m}=-\displaystyle\frac{1}{(a_{ii}+a_{i+m,i+m}-r)}\displaystyle\sum_{p=1}^{m-1}a_{i,i+p}a_{i+p,i+m}=-\displaystyle\frac{1}{(a_{ii}+a_{i+m,i+m}-r)}\alpha\gamma,
$$
so, in this case the $(1,m)$ entries of equation (\ref{eq7}) is
$$
\begin{array}{rcl}
a_{i,i+m}\left(a_{ii}\!+\!a_{i+m,i+m}\!-\!r\right)\!+\!\alpha\gamma\!\!\!\!&\!=\!&\!\!\!\!\left(-\displaystyle\frac{\alpha\gamma}{a_{ii}\!+\!a_{i+m,i+m}\!-\!r}\ \right)\left(a_{ii}\!+\!a_{i+m,i+m}\!-\!r\right)\!+\!\alpha\gamma\\
&&\\
\!\!\!\!&\!=\!&\!\!\!\!0.
\end{array}
$$

Therefore, $A(m,i)$ also satisfies the quadratic equation (\ref{eq1})
\end{itemize}

Thus, we have proved that $A$ satisfies the quadratic equation (\ref{eq1}) in the upper triangular matrix ring $M$ where $a_{ii}\in\{a,b\}$ and $a, b$ are different roots of the equation $x^2-rx+s=0$ if and only if $A$ is described as in (i), (ii) and (iii) of the Theorem \ref{th1}
\end{proof}

Follows immediately from Theorem \ref{th1} the results of Hou \cite{Hou} and Slowik \cite{Slowik}: 

\begin{corollary}[Hou \cite{Hou}] 
We can construct any $n\times n$ idempotent upper triangular matrix over $R$ that has only zeros and ones on its diagonal
\begin{itemize}
    \item[(i)] For all $i$, the entries in the main diagonal  $a_{ii}\in \{0,1\}$.
    \item[(ii)] For $i<j$, if  $a_{ii} = a_{jj}$, then $a_{ij}$ equals to
\begin{equation}\label{eqth3}
a_{ij}=\left\{\begin{array}{lr} 
0\ \ & \text{if}\ \ j=i+1\\
 (1-2a_{ii})\displaystyle\sum_{p=i+1}^{j-1}a_{ip}a_{pj} 
&\text{if}\  \ j > i +1. 
 \end{array}\right.
\end{equation}
\item[(iii)] For $i < j$, if $a_{ii} \not = a_{jj}$, then $a_{ij}$ can be chosen arbitrarily.
\end{itemize}

\end{corollary}
\begin{proof} For $a_{ii}\in\{0,1\}$ the quadratic equation (\ref{eq1}) equals to $A^2=A$ then $A$ is a idempotent matrix.

We need to verify that equation (\ref{eqth3}) of the Theorem \ref{th1} yields the same possibilities for $a_{ij}$ shown in the procedure above. The equation (\ref{eqth}) becomes

$$
\begin{array}{rcl}
a_{ij}&=&-\displaystyle\frac{1}{a_{ii}-{\textit{another  root}}}\displaystyle\sum_{p=i+1}^{j-1}a_{ip}a_{pj} \\
&=& (1-2a_{ii} )\displaystyle\sum_{p=i+1}^{j-1}a_{ip}a_{pj} 
\end{array}
$$
for $a_{ii}\in\{0,1\}$.

\end{proof}
\begin{corollary}[Slowik \cite{Slowik}] We can construct any $n\times n$ involution upper triangular matrix over $R$ wher $a_{ii}\in\{1,-1\}$
\begin{itemize}
    \item[(i)] For all $i$, the entries in the main diagonal  $a_{ii}\in \{-1,1\}$.
    \item[(ii)] For $i<j$, if  $a_{ii} = a_{jj}$, then $a_{ij}$ equals to
\begin{equation}\label{eqth4}
a_{ij}=\left\{\begin{array}{lr} 
0\ \ & \text{if}\ \ j=i+1\\
 -(2a_{ii})^{-1}\displaystyle\sum_{p=i+1}^{j-1}a_{ip}a_{pj} 
&\text{if}\  \ j > i +1. 
 \end{array}\right.
\end{equation}
\item[(iii)] For $i < j$, if $a_{ii}  =- a_{jj}$, then $a_{ij}$ can be chosen arbitrarily.
\end{itemize}

\end{corollary}
\begin{proof} For $a_{ii}\in\{-1,1\}$ the quadratic equation (\ref{eq1}) equals to $A^2=I$ then $A$ is a Involution matrix.

We need to verify that equation (\ref{eqth4}) of the Theorem \ref{th1} yields the same possibilities for $a_{ij}$ shown in the procedure above. The equation (\ref{eqth}) becomes

$$
\begin{array}{rcl}
a_{ij}&=&-\displaystyle\frac{1}{a_{ii}-{\text{another \ root}}}\displaystyle\sum_{p=i+1}^{j-1}a_{ip}a_{pj} \\
&=& -\displaystyle\frac{1}{a_{ii}-(-a_{ii})}\displaystyle\sum_{p=i+1}^{j-1}a_{ip}a_{pj}\\
&=& -\displaystyle\frac{1}{2a_{ii}}\displaystyle\sum_{p=i+1}^{j-1}a_{ip}a_{pj}
\end{array}
$$
for $a_{ii}\in\{-1,1\}$.

\end{proof}
\section{Compute the number of all solutions for the quadratic polynomial equation}

\begin{theorem} Let $R$ be an associative ring with identity 1 and $|R|=q$ the number of the elements
in $R$. Then the total number of $ n\times n$ upper triangular that satisfy the quadratic equaton $A^2-rA+sI=0$  with $a_{ii}\in\{a,b\}$ on
the diagonal where $\{a,b\}$ different roots of the quadratic equation $x^2-rx-s=0$ is equal to
$$ \displaystyle \sum_{\substack{n_1+n_1=n \\ 0\leq n_i}} \binom{n}{n_1n_2}\cdot q^{n_1n_2}.
$$
where $n_1, n_2$ are the number of times  that appears $a,b$ in the diagonal respectively, $r=a+b$ and $s=ab$.
\end{theorem}

\begin{proof}[Proof of Theorem \ref{th2}] By Theorem \ref{th1}, the number of possible upper triangular matrices that satisfy the quadratic equation $A^2-rA+sI=0$ with the set $D=\{a,b\}$, $a\not= b$ on  the diagonal  depends entirely on which pairs of diagonal entries have $a_{ii}\not=a_{jj}$. To enumerate those possibilities, consider an integer column vector $d=(d_1,d_2,\ldots,d_n)$ the respective diagonal having each $d_i\in D$ and denote for $n_1,n_2$ the numbers of $a,b$ that appears in the diagonal respectively, such that $n_1+n_2=n$ with $0\leq n_i$ for $i=1,2$. By $\Delta$ we denote the number de pairs $(d_i,d_j)$ with $i<j$ and $d_i\not= d_j$. Notice that
$$\Delta = n_1\cdot n_2.$$
In particular, $\Delta$ is independent of the order in which the elements of the set $D$ appear on $d$. Consequently we have on the diagonal  yields
$q^{\Delta} = q^{n_1\cdot n_2},$
possible upper tiangular matrices that satisfy the quadratic equation $A^2-rA+sI=0$.

Finally, all $d_i's$ can be put on our main diagonal on
$$\binom{n}{n_1}\binom{n-n_1}{n_2}=\binom{n}{n_1,n_2},$$
where
$$\binom{n}{n_1,n_2}=\frac{n!}{n_1! n_2!}.
$$

Therefore, the total number of $n\times n$ upper tiangular matrices that satisfy the quadratic equation $A^2-rA+sI=0$  with elements the  set $\{a,b\}$ with $a\not=b$ on the diagonal is

$$ \displaystyle \sum_{\substack{n_1+n_1=n \\ 0\leq n_i}} \binom{n}{n_1n_2}\cdot q^{n_1n_2}.
$$

\end{proof}

\end{document}